\documentclass[a4paper]{article}

\usepackage[latin1]{inputenc}
\usepackage[T1]{fontenc}
\usepackage{amssymb,amsmath,amsthm,amscd,mathrsfs}
\usepackage[all]{xy}
\usepackage{color}
\usepackage{float}
\usepackage{array}
\usepackage[margin=3cm]{geometry}

\newtheorem{thm}{Theorem}
\newtheorem{lemme}{Lemma}

\newtheorem{prop}{Proposition}

\theoremstyle{remark}
\newtheorem{rmk}{Remark}

\newtheorem{defi}{Definition}

\newtheorem{nota}{Notation}

\DeclareMathOperator{\Z}{\mathbb{Z}}

\DeclareMathOperator{\Supp}{Supp}

\DeclareMathOperator{\tors}{tors}

\newcommand{\eq}[1][r]
{\ar@<-3pt>@{-}[#1]
\ar@<-1pt>@{}[#1]|<{}="gauche"
\ar@<+0pt>@{}[#1]|-{}="milieu"
\ar@<+1pt>@{}[#1]|>{}="droite"
\ar@/^2pt/@{-}"gauche";"milieu"
\ar@/_2pt/@{-}"milieu";"droite"}

\makeatletter
\newcommand{\opnorm}{\@ifstar\@opnorms\@opnorm}
\newcommand{\@opnorms}[1]{%
  \left|\mkern-1.5mu\left|\mkern-1.5mu\left|
   #1
  \right|\mkern-1.5mu\right|\mkern-1.5mu\right|
}
\newcommand{\@opnorm}[2][]{%
  \mathopen{#1|\mkern-1.5mu#1|\mkern-1.5mu#1|}
  #2
  \mathclose{#1|\mkern-1.5mu#1|\mkern-1.5mu#1|}
}
\makeatother

\begin{document}
\title{\emph{Erratum}:\ \bf  Integral cohomology of the Generalized Kummer fourfold}

\author{Simon \textsc{Kapfer}; Grégoire \textsc{Menet}} 

\maketitle
It was pointed out by B. Totaro that the reference used for \cite[Theorem 5.2]{Original} is inappropriate. 
The latter concern the torsion of the integral cohomology of the generalized Kummer. 
Here we show that \cite[Theorem 5.2]{Original} holds at least in dimension 4.
All the other results of \cite{Original} remain unaffected.
It would also be interesting to find out whether the generalized Kummer varieties of higher dimension have torsion-free cohomology or not.
 \begin{center}
 \begin{large}
 \textbf{The integral cohomology of the generalized Kummer fourfold is torsion free}
 \end{large}
 \end{center}
Let $A$ be a 2-dimensional complex torus. Let $A^{[3]}$ be the Hilbert scheme of 3 points on $A$ and $s:A^{[3]}\rightarrow A$ the summation morphism.  
The generalized Kummer fourfold is defined by $K_2(A):=s^{-1}(0)$.
We can also consider the following embedding: $j:A\times A\hookrightarrow A\times A \times A: (x,y)\rightarrow (x,y,-x-y)$.
The action of the symmetric group $\mathfrak{S}_3$ on $A\times A \times A$ provides an action on $A\times A$ via the embedding $j$.
Then $K_2(A)$ can also be seen as a resolution of $\left(A\times A\right)/\mathfrak{S}_3$. 
\begin{thm}\label{main}
The cohomology $H^*(K_2(A),\Z)$ is torsion free.
\end{thm}
\begin{rmk}
We denote by $\tors$ the torsion of groups.
Because of the Poincaré duality and the universal coefficient theorem, we have:
\begin{align*}
\tors H^4=\tors H_4&=\tors H^5,& \tors H^6=\tors H_2&=\tors H^3,& \tors H^7=\tors H_1&=\tors H^2=0.
\end{align*}
Thus, it suffices to prove that $H^3(K_2(A),\Z)$ and $H^5(K_2(A),\Z)$ are torsion free. Moreover, since Theorem \ref{main} is only a topological result, without loss of generality, we can assume that $A$ is an abelian surface.
\end{rmk}
Let
$W_\tau\subset K_2(A)$ be the locus of subschemes supported at $\tau\in A[3]$.
As it is explained in \cite[Section 4]{Hassett}, we have:
\begin{equation}
W_\tau\simeq \mathbb{P}(1,1,3).
\label{Briançon}
\end{equation}
Let  $pt_\tau \in W_\tau$ be the singular point 
and $W_\tau^*:=W_\tau\smallsetminus pt_\tau$.
Put
\begin{align*}
U&:=K_2(A)\setminus \bigcup_{\tau\in A[3]} W_\tau, &
U'&:=K_2(A)\setminus \bigcup_{\tau\in A[3]} pt_\tau.
\end{align*}
\begin{lemme}
We have $\tors H^3(U,\Z)=\tors H^3(K_2(A),\Z)$ and an injection $\tors H^5(K_2(A),\Z)\hookrightarrow\tors H^5(U,\Z)$.
\end{lemme}
\begin{proof}
The generalized Kummer $K_2(A)$ is smooth in $pt_\tau$. Hence, applying Thom's isomorphism to the long exact sequence of the relative cohomology
of the pair $(K_2(A),U')$,
we obtain:
\begin{equation}
H^3(K_2(A),\Z)=H^3(U',\Z)\ \text{and}\ H^5(K_2(A),\Z)=H^5(U',\Z).
\label{U'}
\end{equation}
Moreover:
\begin{equation}
\xymatrix@C7pt@R0pt{ H^3(U',U,\Z)\ar[r] & H^{3}(U',\Z)\ar[r] & H^{3}(U,\Z)\ar[r]&H^{4}(U',U,\Z)\ar[r] & H^{4}(U',\Z)\ar[r] & \\
\ar[r]&H^{4}(U,\Z)\ar[r]&H^{5}(U',U,\Z)\ar[r]& H^{5}(U',\Z)\ar[r] & H^{5}(U,\Z)\ar[r]&H^{6}(U',U,\Z)
.}
\label{exact1}
\end{equation}
By Thom's isomorphism, $H^3(U',U,\Z)=0$ and $H^{4}(U',U,\Z)=\bigoplus_{\tau\in A[3]}H^0(W_\tau^*,\Z)$ is torsion free.
Furthermore, $H^{5}(U',U,\Z)=\bigoplus_{\tau\in A[3]}H^1(W_\tau^*,\Z)=0$ since, by (\ref{Briançon}), $W_\tau^*$ is the quotient of $\mathbb{P}^2\smallsetminus \left\{pt\right\}$ by an automorphism of order 3. 
It follows from (\ref{exact1}):
$$\tors H^3(U',\Z)=\tors H^3(U,\Z)\ \text{and}\ \tors H^5(U',\Z)\hookrightarrow\tors H^5(U,\Z).$$
Then (\ref{U'}) concludes the proof.
\end{proof}
Hence, it remains to prove that $H^3(U,\Z)$ and $H^5(U,\Z)$ are torsion free. 
To do so, we consider
$$V:=(A\times A)\ \smallsetminus\  A[3],$$
where $A[3]$ is embedded in $A\times A$ diagonally: $A[3]\hookrightarrow A\times A: x\rightarrow (x,x)$.

Let $r:\widetilde{V}\rightarrow V$ be the blow-up of $V$ in 
\begin{align*}
\Delta&:=\left\{\left.(x,x)\in V\right|\ x\in A\right\}, &
S_1 &:=\left\{\left.(x,-2x)\in V\right|\ x\in A\right\}, &
S_2&:=\left\{\left.(-2x,x)\in V\right|\ x\in A\right\}.
\end{align*}
As explained in \cite[Section 7]{Beauville}, 
we have: 
\begin{equation}
U\simeq \widetilde{V}/\mathfrak{S}_3.
\label{Beauvilleequa}
\end{equation}
\begin{lemme}\label{2torsion}
The groups $H^3(U,\Z)$ and $H^5(U,\Z)$ can only have 3-torsion. 
\end{lemme}
\begin{proof}
Let $$S':=\left\{\left.\xi\in A^{[2]}\right|\ \Supp \xi=\left\{x,-2x\right\},\ x\in A\right\}.$$
The surface $S'$ is isomorphic to the blow-up of $A$ in $A[3]$.
We consider $r:\widetilde{A^{[2]}}\rightarrow A^{[2]}$ the blow-up of $A^{[2]}$ in $S'$.
For $\tau\in A[3]$, we denote $$\Sigma_\tau:=r^{-1}\left(\left\{\left.\xi\in A^{[2]}\right|\ \Supp \xi=\left\{\tau\right\}\right\}\right).$$
The surfaces $\Sigma_\tau$ are Hirzebruch surfaces.
We also consider 
$$\widetilde{\mathcal{W}}:=\widetilde{A^{[2]}}\smallsetminus \bigcup_{\tau\in A[3]}\Sigma_\tau.$$
We have:
\begin{equation}
\widetilde{\mathcal{W}}\simeq \widetilde{V}/\mathfrak{S}_2.
\label{key2}
\end{equation}
Indeed, if we denote by $\widetilde{A\times A}$ the blow-up of $A\times A$ in the diagonal, it is well known that $\widetilde{A\times A}/\mathfrak{S}_2\simeq A^{[2]}$. 
For $\tau\in A[3]$, we denote $\ell_\tau:=\left\{\left.\xi\in A^{[2]}\right|\ \Supp \xi=\left\{\tau\right\}\right\}$.
Then, if we consider $V_1$ the blow-up of $V$ in $\Delta$, we have $V_1/\mathfrak{S}_2\simeq A^{[2]}\smallsetminus \bigcup_{\tau\in A[3]} \ell_\tau$.  
Therefore
by \cite[Corollary II 7.15]{Hartshorne}, we obtain a commutative diagram:
$$\xymatrix{\widetilde{V}\ar[r]\ar[d]&V_1\ar[d]\\
\widetilde{\mathcal{W}}\ar[r]&V_1/\mathfrak{S}_2.
}$$
This provides (\ref{key2}).
Then, it follows from (\ref{Beauvilleequa}) a triple cover:
$$\pi:\widetilde{\mathcal{W}}\rightarrow U.$$
Hence, if we prove that $H^3(\widetilde{\mathcal{W}},\Z)$ and $H^5(\widetilde{\mathcal{W}},\Z)$ are torsion free, Lemma \ref{2torsion} will be proven by \cite[Theorem 5.4 ]{Transfers}.

We know that $H^*(A^{[2]},\Z)$ is torsion free from \cite[Theorem 2.2]{Totaro}. 
Then, we deduce from \cite[Theorem 7.31]{Voisin} (or from \cite[Theorem 4.1]{Li} which is more general) that:
\begin{equation}
\tors H^*\left(\widetilde{A^{[2]}},\Z\right)=0.
\label{torsionEcla}
\end{equation}
Consider the exact sequence:
\begin{equation}
\xymatrix@C7pt@R0pt{ H^3(\widetilde{A^{[2]}},\widetilde{\mathcal{W}},\Z)\ar[r] & H^{3}(\widetilde{A^{[2]}},\Z)\ar[r] & H^{3}(\widetilde{\mathcal{W}},\Z)\ar[r]&H^{4}(\widetilde{A^{[2]}},\widetilde{\mathcal{W}},\Z)\ar[r] & H^{4}(\widetilde{A^{[2]}},\Z)\ar[r] & \\
\ar[r]&H^{4}(\widetilde{\mathcal{W}},\Z)\ar[r]&H^{5}(\widetilde{A^{[2]}},\widetilde{\mathcal{W}},\Z)\ar[r]& H^{5}(\widetilde{A^{[2]}},\Z)\ar[r] & H^{5}(\widetilde{\mathcal{W}},\Z)\ar[r]&H^{6}(\widetilde{A^{[2]}},\widetilde{\mathcal{W}},\Z)
.}
\label{exact2}
\end{equation}
By Thom's isomorphism:
$$H^k(\widetilde{A^{[2]}},\widetilde{\mathcal{W}},\Z)\simeq \bigoplus_{\tau\in A[3]} H^{k-4}(\Sigma_\tau,\Z).$$
Since $\Sigma_\tau$ is an Hirzebruch surface, $H^3(\widetilde{A^{[2]}},\widetilde{\mathcal{W}},\Z)=H^5(\widetilde{A^{[2]}},\widetilde{\mathcal{W}},\Z)=0$
and $H^{4}(\widetilde{A^{[2]}},\widetilde{\mathcal{W}},\Z)$, $H^{6}(\widetilde{A^{[2]}},\widetilde{\mathcal{W}},\Z)$ are torsion free. 
It follows from (\ref{torsionEcla}) and (\ref{exact2}) that $H^3(\widetilde{\mathcal{W}},\Z)$ and $H^5(\widetilde{\mathcal{W}},\Z)$ are torsion free.
\end{proof}
We can also consider the double cover:
$$\widetilde{V}/\mathfrak{A}_3\rightarrow U.$$
Applying \cite[Theorem 5.4 ]{Transfers} and Lemma \ref{2torsion}, it is suffices to prove that $H^3(\widetilde{V}/\mathfrak{A}_3,\Z)$ and $H^5(\widetilde{V}/\mathfrak{A}_3,\Z)$
are torsion free to conclude the proof of Theorem \ref{main}. 
\begin{lemme}\label{last}
The groups $H^3(\widetilde{V}/\mathfrak{A}_3,\Z)$ and $H^5(\widetilde{V}/\mathfrak{A}_3,\Z)$ are torsion free.
\end{lemme}
First, we show that $\tors H^3(V/\mathfrak{A}_3,\Z)=\tors H^5(V/\mathfrak{A}_3,\Z)=0$.
Since the action of $\mathfrak{A}_3$ on $V$ is free, it can be realized using the equivariant cohomology as explained in \cite[Section 4]{Lol}.
The computation of the equivariant cohomology can be done using the Boissière--Sarti--Nieper-Wisskirchen invariants defined in \cite[Section 2]{SmithTh}. 
We recall their definition in our specific case. Let $T$ be a $3$-torsion-free $\Z$-module of finite rank equipped with a linear action of $\mathfrak{A}_3=\left\langle \sigma_{1,2,3}\right\rangle$. 
We consider the action of $\mathfrak{A}_3$ on $T\otimes\mathbb{F}_3$. Then the matrix of the endomorphism $\sigma_{1,2,3}$ on $T\otimes\mathbb{F}_3$ admits a Jordan normal form.  
We can decompose $T\otimes\mathbb{F}_3$ as a direct sum of some $\mathbb{F}_3[\mathfrak{A}_3]$-modules $N_{q}$ of dimension $q$ with $1\leq q\leq 3$, where $\sigma_{1,2,3}$ acts on the $N_{q}$ in a suitable basis, respectively by matrices of the following form:
$$\begin{pmatrix}
1\end{pmatrix},\ 
\begin{pmatrix}
1 & 1\\
0 & 1
\end{pmatrix}\
\text{and}\
\begin{pmatrix}
1 & 1 & 0\\
0 & 1 & 1\\
0 & 0 & 1
\end{pmatrix}.
$$
\begin{defi}
We define the integer $\ell_{q}(T)$ as the number of blocks of size $q$ in the Jordan decomposition of the $\mathbb{F}_3[\mathfrak{A}_3]$-module $T\otimes\mathbb{F}_3$, so that $T\otimes\mathbb{F}_3\simeq \bigoplus_{q=1}^{3} N_{q}^{\oplus \ell_{q}(T)}$. 
\end{defi}
\begin{nota}
Let $X$ being $V$ or $A\times A$.
For all $0\leq k\leq \dim X$ and all $q\in\left\{1,...,3\right\}$, we denote:
$$ \ell_q^k(X)=\ell_q(H^k(X,\Z)).$$

\end{nota}
\begin{prop}\label{AA}
The Boissière--Sarti--Nieper-Wisskirchen invariants for the $\mathfrak{A}_3$-action on $A\times A$ are:
\begin{itemize}
\item[(i)]
$\ell_1^1(A\times A)=\ell_3^1(A\times A)=0$ and $\ell_2^1(A\times A)=4$;
\item[(ii)]
$\ell_1^2(A\times A)=10$, $\ell_2^2(A\times A)=0$ and  $\ell_3^2(A\times A)=6$;
\item[(iii)]
$\ell_1^3(A\times A)=0$, $\ell_2^3(A\times A)=16$ and  $\ell_3^3(A\times A)=8$;
\item[(iv)]
$\ell_1^4(A\times A)=19$, $\ell_2^4(A\times A)=0$ and  $\ell_3^4(A\times A)=17$.
\end{itemize}
\end{prop}
\begin{proof}
We can start by calculating the $\ell_q^1(A\times A)$, $q\in \left\{1,2,3\right\}$. Applying the same idea as \cite[Lemma 6.14]{SmithTh}, we deduce the other invariants $\ell_q^k(A\times A)$ using the fact that: 
$$H^k(A\times A,\Z)\simeq\wedge^k H^1(A\times A,\Z),$$
for all $k$. 
\begin{itemize}
\item[(i)]
For $a\in H^1(A,\Z)$ and a generator $A$ of $H^0(A,\Z)$,
$\sigma_{1,2,3}^*(A\otimes a)=-a\otimes A,\ \text{and}\ \sigma_{1,2,3}^*(a\otimes A)=-A\otimes a+a\otimes A.$
It follows (i).
\item[(ii)]
We have
$\wedge^2 (N_2^4)=(\wedge^2 N_2)^4\oplus(N_2\otimes N_2)^6=N_1^4\oplus(N_3\oplus N_1)^6.$
\item[(iii)]
We have
$\wedge^3 (N_2^4)=(\underbrace{N_2\otimes N_2\otimes N_2}_{(1)})^4\oplus (\underbrace{\wedge^2 N_2\otimes N_2 }_{(2)})^{12}.$

Moreover,
$\begin{cases}
(1)=(N_3\oplus N_1)\otimes N_2=N_3^2\oplus N_2,\\
(2)=N_1\otimes N_2=N_2.
\end{cases}$
\item[(iv)]
We have
$\wedge^4 (N_2^4)=\underbrace{N_2\otimes N_2\otimes N_2\otimes N_2}_{(1)}\oplus(\underbrace{\wedge^2 N_2 \otimes \wedge^2 N_2}_{(2)})^6\oplus (\underbrace{\wedge^2 N_2\otimes N_2 \otimes N_2}_{(3)})^{12}.$

Moreover:
\begin{equation}
\begin{cases}
(1)=(N_3\oplus N_1)\otimes (N_3\oplus N_1)=\overbrace{N_3\otimes N_3}^{N_3^3} \oplus N_3^2\oplus N_1=N_3^5\oplus N_1,\\
(2)=N_1,\\
(3)=N_1\otimes (N_1\oplus N_3)=N_1\oplus N_3.
\end{cases}
\end{equation}
\end{itemize}
\end{proof}
\begin{proof}[Proof of Lemma \ref{last}]
Since $H^k(A\times A,\Z)\simeq H^k(V,\Z)$ for all $k\leq 6$, $\ell_q^k(V)=\ell_q^k(A\times A)$ for all $q\in \left\{1,2,3\right\}$ and $k\leq 6$. 
Hence by \cite[Corollary 4.2]{Lol} and Proposition \ref{AA}, we have $H^{p}(\mathfrak{A}_3,H^q(V,\Z))=0$ for all $p+q=3$, $p\neq 0$, and $p+q=5$, $p\neq 0$.

Since $\mathfrak{A}_3$ acts freely on $V$, the spectral sequence of equivariant cohomology provides that $H^3(V/\mathfrak{A}_3,\Z)$ and $H^5(V/\mathfrak{A}_3,\Z)$ are torsion free (see \cite[Section 4]{Lol} for a reminder about this spectral sequence). 
Moreover, $\widetilde{V}/\mathfrak{A}_3$ is the blow-up of $V/\mathfrak{A}_3$ in the image $\overline{\Delta}$ of $\Delta$ in $V/\mathfrak{A}_3$. 
Since $V/\mathfrak{A}_3$ and $\overline{\Delta}$ are smooth, by  \cite[Theorem 4.1]{Li} we have an isomorphism of graded $\Z$-modules:
$$H^*(\widetilde{V},\Z)\simeq H^*(V,\Z)\oplus t\cdot H^*(\overline{\Delta},\Z),$$
where $t$ is a class of degree 2. Hence, in degree 3 and 5, we obtain:
$$\tors H^3(\widetilde{V}/\mathfrak{A}_3,\Z)=\tors H^3(V/\mathfrak{A}_3,\Z)\oplus \tors H^1(\overline{\Delta},\Z)=0,$$ and 
$$\tors H^5(\widetilde{V}/\mathfrak{A}_3,\Z)=\tors H^5(V/\mathfrak{A}_3,\Z)\oplus \tors H^3(\overline{\Delta},\Z)=\tors H^3(\overline{\Delta},\Z).$$
It remains to show that $H^3(\overline{\Delta},\Z)$ is torsion free. Since $\overline{\Delta}\simeq A\smallsetminus A[3]$, it can be done considering the following exact sequence:
$$\xymatrix{ 0\ar[r] & H^3(A,\Z)\ar[r] & H^{3}(A\smallsetminus A[3],\Z)\ar[r]&H^{4}(A,A\smallsetminus A[3],\Z)
,}$$
where $H^3(A,\Z)$ is torsion free and $H^{4}(A,A\smallsetminus A[3],\Z)\simeq \Z^{81}$ by Thom's isomorphism.
\end{proof}
~\\
\textbf{Acknowledgements.}
This work was supported by the ERC-ALKAGE, grant No. 670846.
\bibliographystyle{amssort}

\begin{thebibliography}{10}
\bibitem{Transfers}
M.A.~Aguilar, C.~Prieto,
\newblock {\em Transfers for ramified covering maps in homology and cohomology}.
\newblock Int. J. Math. Sci.,
\newblock (2006), Art. ID 94651, 28pp.

\bibitem{Beauville}
A.~Beauville,
\newblock {\em Variétés kähleriennes dont la première classe de chern est nulle},
\newblock J. Differential geometry,
\newblock 18 (1983) 755-782.

\bibitem{SmithTh}
S.~Boissi\`ere, M.~Nieper-Wisskirchen, A.~Sarti,
\newblock {\em Smith theory and irreducible holomorphic symplectic manifolds},
\newblock J. Topo. 6 (2013), no. 2, 361-390.

\bibitem{Li}
D.~Haibao, L.~Banghe,
\newblock {\em Topology of Blow-ups and Enumerative Geometry},
\newblock arXiv:0906.4152.

\bibitem{Hartshorne}
R.~Hartshorne,
\newblock {\em Algebraic Geometry},
\newblock Graduate Texts in Mathematics.

\bibitem{Hassett}
B.~Hassett and Y.~Tschinkel,
\newblock {\em Hodge theory and Lagrangian planes on generalized Kummer fourfolds},
\newblock  Moscow Math. Journal, 13, no. 1, 33-56, (2013).

\bibitem{Original}
S.~Kapfer and G.~Menet,
\newblock {\em Integral cohomology of the Generalized Kummer fourfold},
\newblock  Algebraic Geometry, 5 (5) (2018) 523-567.

\bibitem{Lol}
G.~Menet,
\newblock {\em On the integral cohomology of quotients of complex manifolds},
\newblock Journal de Mathématiques pures et appliquées, 
\newblock 119 (2018), no.9, 280-325.

\bibitem{Totaro}
B.~Totaro,
\newblock {\em The integral cohomology of the Hilbert scheme of two points},
\newblock Forum Math. Sigma 4 (2016).

\bibitem{Voisin}
C.~Voisin,
\newblock{\em Hodge Theory and Complex Algebraic Geometry. I,II},
\newblock Cambridge Stud. Adv. Math.,
\newblock 76, 77, Cambridge Univ. Press, 2003.
\\

\end{thebibliography}

\noindent
Simon \textsc{Kapfer}

\noindent
Institute of Mathematics of Augsburg university, D--86159 Augsburg, Germany

\noindent
{\tt kapfer.simon@freenet.de}
~\\

\noindent
Gr\'egoire \textsc{Menet}

\noindent
Institut Fourier, 100 rue des Mathématiques, 38610 Gières, France

\noindent
{\tt gregoire.menet@univ-grenoble-alpes.fr}

\end{document}